\documentclass[12pt]{amsart}

\usepackage{amsthm,amssymb,amsmath,amsfonts, tikz, graphicx, a4wide,subfigure,bm}
\usepackage[english]{babel}
\usepackage{subfigure}
\usepackage{mathtools}
\usepackage{url}
\usepackage{enumerate}
\usepackage[hidelinks]{hyperref}

\usepackage{xargs} 
\usepackage{fancyhdr}  
\usepackage{xcolor}
\usepackage[colorinlistoftodos,prependcaption]{todonotes}
\newcommandx{\unsure}[2][1=]{\todo[linecolor=red,backgroundcolor=red!25,bordercolor=red,#1]{#2}}
\newcommandx{\change}[2][1=]{\todo[linecolor=blue,backgroundcolor=blue!25,bordercolor=blue,#1]{#2}}
\newcommandx{\info}[2][1=]{\todo[linecolor=OliveGreen,backgroundcolor=OliveGreen!25,bordercolor=OliveGreen,#1]{#2}}
\newcommandx{\improvement}[2][1=]{\todo[linecolor=Plum,backgroundcolor=red!25,bordercolor=red,#1]{#2}}
\newcommandx{\thiswillnotshow}[2][1=]{\todo[disable,#1]{#2}}

\begin{document}

\newtheorem{prop}{Proposition}[section]
\newtheorem{theorem}{Theorem}[section]
\newtheorem{lemma}{Lemma}[section]
\newtheorem{cor}{Corollary}[section]
\newtheorem{remark}{Remark}[section]
\theoremstyle{definition}
\newtheorem{defn}{Definition}[section]
\newtheorem{ex}{Example}[section]

\numberwithin{equation}{section}

\title{Intermittency generated by attracting and weakly repelling fixed points}

\author[Zeegers]{Benthen Zeegers}

%
%
%
%
%
\address{B.P. Zeegers\\ Mathematical Institute, University of Leiden, PO Box 9512, 2300 RA Leiden, The Netherlands}
\email{b.p.zeegers@math.leidenuniv.nl}

\date{Version of \today}

\begin{abstract} Recently for a class of critically intermittent random systems a phase transition was found for the finiteness of the absolutely continuous invariant measure. The systems for which this result holds are characterized by the interplay between a superexponentially attracting fixed point and an exponentially repelling fixed point. In this article we consider a closely related family of random systems with instead exponentially fast attraction to and polynomially fast repulsion from two fixed points, and show that such a phase transition still exists. The method of the proof however is different and relies on the construction of a suitable invariant set for the transfer operator.
\end{abstract}

\maketitle

\section{Introduction}\label{sec1}

Intermittent dynamical systems are systems that fluctuate between spending long periods in a chaotic state and long periods in a seemingly steady state. Well-known examples of one-dimensional intermittent dynamical systems are the LSV maps from \cite{LSV} given by
\begin{align}\label{eq1.1}
S_{\alpha}: [0,1] \to [0,1], \quad S_{\alpha}(x) =  \begin{cases}
x(1+2^{\alpha} x^{\alpha}) & \text{if}\quad x \in [0,\frac{1}{2}],\\
2x-1 & \text{if}\quad x \in (\frac{1}{2},1],
\end{cases}
\end{align}
where $\alpha > 0$. These maps were introduced as a simplification of the Manneville-Pomeau maps on $[0,1]$ given by $x \mapsto x+x^{1+\alpha} \bmod 1$ with $\alpha > 0$ which were considered to study intermittency in the context of transition to turbulence in convective fluids, see \cite{ManPum,ManPum2,BPV}. For the LSV maps and Manneville-Pomeau maps the periods of chaotic behaviour are caused by the uniform expansion of the maps away from zero whereas the neutral fixed point at zero makes orbits spend a long time close to zero.

\medskip 
In the recent papers \cite{AbbGhaHom,HP,Zeegers21,Zeegers22} \emph{critically} intermittent dynamical systems are studied. These are systems that exhibit intermittency coming from the interplay between a superattracting fixed point and a repelling fixed point. More specifically, in \cite{Zeegers21,Zeegers22} random dynamical systems on $[0,1]$ are analysed that generate i.i.d.~random compositions of so-called good bad and bad maps. The bad maps share a superstable fixed point $c \in (0,1)$ with $(0,1)$ as basin of attraction and the good maps send $c$ into $\{0,1\}$, which is a repelling invariant set for both the good and bad maps. The random orbits then converge superexponentially fast to the point $c$ under iterations of the bad maps, and once a good map is applied then diverge exponentially fast from $\{0,1\}$. This is illustrated in Figure \ref{fig1}(a) with the logistic maps $T_2(x) = 2x(1-x)$ and $T_4(x)=4x(1-x)$. It was shown in \cite{Zeegers21,Zeegers22} that when varying the probabilities of chosing the good and bad maps these random systems exhibit a phase transition where the unique absolutely continuous invariant measure changes from finite to infinite. 

\begin{figure}[h] \label{fig1}
\centering
\subfigure[]
{
\begin{tikzpicture}[scale =4]
\draw(-.01,0)--(1.01,0)(0,-.01)--(0,1.01);
\draw[dotted](.5,0)--(.5,1)(0,1)--(.5,1)(.5,.5)--(0,.5);
\draw[dotted](0,0)--(1,1);
\draw[line width=.4mm, green!50!blue!70!black, smooth, samples =20, domain=0:1] plot(\x, { 4* \x * (1-\x)});
\draw[line width=.4mm, red!40!blue!70!black, smooth, samples =20, domain=0:1] plot(\x, { 2* \x * (1-\x)});
\draw[red, dashed, line width=.25mm](.25,0)--(.255,.375)--(.375,.375)--(.37995,.4688)--(.4688,.4688)--(.4688,.9961)--(.9961,.9961)--(.9961,.0155)--(.0155,.0155)--(.0155,.0262)--(.0262,.0262)--(0.0262,.051)--(0.051,.051)--(0.051,.1975)--(0.1975,.1975);

\node[below] at (.25,0){\small $x$};
\node[below] at (-.01,0){\tiny 0};
\node[below] at (1,0){\tiny 1};
\node[below] at (.5,0){\tiny $\frac12$};
\node[left] at (0,1){\tiny 1};
\node[left] at (0,.5){\tiny $\frac12$};
\end{tikzpicture}}
\hspace{1cm}
\subfigure[]
{
\begin{tikzpicture}[scale =4]
\draw(-.01,0)--(1.01,0)(0,-.01)--(0,1.01);
\draw[dotted](.5,0)--(.5,1)(0,1)--(1,1);
\draw[line width=.4mm, green!50!blue!70!black](.5,0)--(1,1);
\draw[line width=.4mm, red!40!blue!70!black, smooth, samples =20, domain=0.5:1] plot(\x, { 0.5+0.3*(\x-0.5)+2*(1-0.3)*(\x-0.5)^2});
\node[below] at (1,0){\tiny 1};
\node[below] at (.5,0){\tiny $\frac12$};
\node[below] at (-.01,0){\tiny 0};
\node[left] at (0,1){\tiny 1};
\draw[dotted](.5,0)--(.5,.5)(.5,.5)--(0,.5);
\node[left] at (0,.5){\tiny $\frac12$};
\draw[line width=.4mm, green!50!blue!70!black, smooth, samples =20, domain=0:0.5] plot(\x, { \x*(1+2^(0.5)*\x^(0.5)});
\draw[line width=.4mm, red!40!blue!70!black, smooth, samples =20, domain=0:0.5] plot(\x, { \x*(1+2^2*\x^2});
\draw[dotted](0,0)--(1,1);
\node[below] at (.73766,0){\small $x$};
\draw[red, dashed, line width=.25mm](.73766,0)--(.73766,.65037)--(.65037,.65037)--(.65037,.57677)--(.57677,.57677)--(.57677,.53128)--(.53128,.53128)--(.53128,.510752)--(.510752,.510752)--(.510752,.021504)--(.021504,.025964)--(0.025964,.025964)--(0.025964,.0318806)--(0.0318806,.0318806)--(0.0318806,.0399308)--(0.0399308,.0399308)--(.0399308,.051215)--(.051215,.051215)--(.051215,.067606)--(.067606,.067606)--(.067606,.092466)--(.092466,.092466)--(0.092466,.13223)--(0.13223,.13223)--(0.13223,.20023)--(0.20023,.20023)--(0.20023,.213089)--(.213089,.213089)--(.213089,.352198)--(.352198,.352198);
\end{tikzpicture}}
\caption{Intermittency in the random system of (a) the logistic maps $T_2,T_4$ (b) the LSV map $S_{\alpha}$ with $\alpha=0.5$ and the map $R_{\beta,K}$ from \eqref{eq1.2} with $\beta=2$ and $K=0.4$. The dashed lines indicates part of a random orbit of $x$.}
\label{fig:interm}
\end{figure}
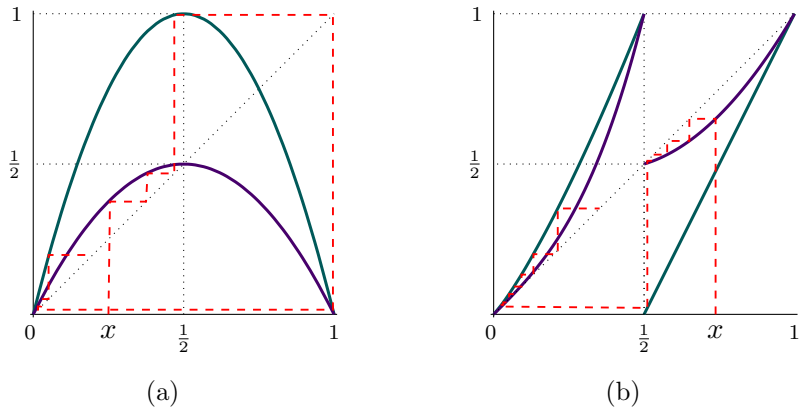

\medskip
In \cite{Zeegers21} the question was asked what happens to the absolutely continuous invariant measure, if it exists, when the superexponential convergence to $c$ is replaced by exponential convergence to $c$ and the exponential divergence from $0$ and $1$ is replaced by polynomial divergence from $0$ and $1$. In this article we investigate this by considering a random system that generates i.i.d.~random compositions of a finite fixed number of maps of two types: Type 1 consists of the LSV maps from \eqref{eq1.1} and type 2 consists of LSV maps where the right branch is replaced by increasing branches that map $(\frac{1}{2},1]$ to itself and for which the derivative close to $\frac{1}{2}$ is smaller than 1. The random orbits then converge exponentially fast to $\frac{1}{2}$ under applications of maps of type 2, and as soon as a map of type 1 is applied then diverge polynomially fast from $0$, see Figure \ref{fig1}(b). We will show that such random systems exhibit a phase transition similar to the one found in \cite{Zeegers21,Zeegers22} in the sense that it depends on the features of the maps as well as on the probabilities of choosing the maps whether the system admits a finite absolutely continuous invariant measure or not.

\medskip
The LSV maps have been studied extensively over the past two decades as being the standard one-dimensional example of an intermittent dynamical system. It is well-known that an LSV map $S_{\alpha}$ has a unique absolutely continuous invariant measure that is finite if $\alpha \in (0,1)$ and infinite but $\sigma$-finite if $\alpha \geq 1$, see e.g.~\cite{pianigiani1980,LSV,PS91}. In \cite{BBD14,BB16,zeegers18,NTF18,BBR19,BQT21,NPT21} random systems are studied that generate i.i.d.~random compositions of LSV maps $S_{\alpha}$ where $\alpha$ is sampled from some fixed subset $A \subseteq (0, \infty)$. It is proven in \cite{BBD14} by means of a Young tower that in case $A$ is finite and a subset of $(0,1]$ an absolutely continuous invariant probability measure exists if the minimal value $\alpha_{\min}$ of $A$ lies in $(0,1)$. This was later shown in \cite{zeegers18} as well without the restriction $A \subseteq (0,1]$ as long as $A$ is finite, $\alpha_{\min}$ lies in $(0,1)$ and $\alpha_{\min}$ has strictly positive probability to be sampled. Here the approach of \cite{LSV} is followed by constructing a suitable invariant set for the transfer operator, see Section \ref{sec4}. Recently it has been shown in \cite{BQT21} using renewal theory of operators that the finiteness condition on $A$ can be dropped as well to show the existence of an absolutely continuous invariant probability measure.

\medskip
We define the class $\mathfrak{S} = \{S_{\alpha}: \alpha \in (0,\infty)\}$ where $S_{\alpha}$ is the LSV map from \eqref{eq1.1}, and the class $\mathfrak{R} = \{R_{\alpha,K}: \alpha \in (0,\infty), K \in (0,1) \}$ where
\begin{align}\label{eq1.2}
R_{\alpha,K}(x) =  \begin{cases}
x(1+2^{\alpha} x^{\alpha}) & \text{if}\quad x \in [0,\frac{1}{2}],\\
\frac{1}{2}+K(x-\frac{1}{2})+2(1-K)(x-\frac{1}{2})^2 & \text{if}\quad x \in (\frac{1}{2},1].
\end{cases}
\end{align}
See Figure \ref{fig1}(b). The right branch of $R_{\alpha,K}$ is defined in such a way that $\frac{1}{2}$ and 1 are fixed points for $R_{\alpha,K}$ and that under $R_{\alpha,K}$ orbits eventually approach $\frac{1}{2}$ from above. The rate of this convergence to $\frac{1}{2}$ is determined by $K$. Let $T_1,\ldots,T_N \in \mathfrak{S} \cup \mathfrak{R}$ be a finite collection. We write 
\begin{align*}
&\Sigma_S = \{1 \leq j \leq N: T_j \in \mathfrak{S}\},\\
&\Sigma_R = \{1 \leq j \leq N: T_j \in \mathfrak{R}\},\\
&\Sigma = \{1,\ldots,N\} = \Sigma_S \cup \Sigma_R.
\end{align*}
We assume that $\Sigma_S,\Sigma_R \neq \emptyset$. For each $j \in \Sigma$ we write $\alpha_j \in (0,\infty)$ if $T_j(x) = x(1+2^{\alpha_j} x^{\alpha_j})$ for $x \in [0,\frac{1}{2}]$. For $j \in \Sigma_R$ we moreover write $K_j \in (0,1)$ if $T_j(x) = \frac{1}{2} + K_j(x-\frac{1}{2}) + 2(1-K_j)(x-\frac{1}{2})^2$ for $x \in (\frac{1}{2},1]$.

\medskip
We define the skew product $F$ by
\begin{equation}\label{q:skewproduct}
 F:\Sigma^{\mathbb N} \times [0,1] \to \Sigma^{\mathbb N} \times [0,1], \, (\omega,x) \mapsto (\sigma \omega, T_{\omega_1}(x)),
 \end{equation}
where $\sigma$ denotes the left shift on sequences in $\Sigma^{\mathbb N}$. Let $\mathbf p = (p_j)_{j \in \Sigma}$ be a probability vector with strictly positive entries representing the probabilities with which we choose the maps $T_j$ ($j \in \Sigma$). We write $\mathbb{P}$ for the $\mathbf p$-Bernoulli measure on $\Sigma^\mathbb N$. By drawing $\omega$ from $\Sigma^{\mathbb{N}}$ according to $\mathbb{P}$ iterations under $F$ produce in the second coordinate random orbits in $[0,1]$. Since each of the maps $T_j$ ($j \in \Sigma$) has zero as a neutral fixed point, these random orbits exhibit intermittent behaviour in the sense that periods of chaotic behaviour are followed by periods of spending time near zero. The periods near zero get longer and more frequent for larger values of $p_j$ ($j \in \Sigma_R$), smaller values of $K_j$ ($j \in \Sigma_R$) and larger values of $\alpha_j$ ($j \in \Sigma$). See Figure \ref{fig1}(b).

\medskip
We will consider measures of the form $\mathbb P \times \mu_{\mathbf p}$, where $\mathbb P$ is the $\mathbf p$-Bernoulli measure on $\Sigma^\mathbb N$ and $\mu_{\mathbf p}$ is a Borel measure on $[0,1]$ absolutely continuous with respect to the Lebesgue measure $\lambda$ on $[0,1]$ and satisfying
\begin{align*}
\sum_{j \in \Sigma} p_j \mu_{\mathbf p}(T_j^{-1}A) = \mu_{\mathbf p}(A), \qquad \text{for all Borel sets $A \subseteq [0,1]$}.
\end{align*}
In this case $\mathbb P \times \mu_{\mathbf p}$ is an invariant measure for $F$ and we say that $\mu_{\mathbf p}$ is a {\em stationary} measure for $F$. If $\mu_{\mathbf p}$ is furthermore absolutely continuous with respect to $\lambda$, then we call $\mu_{\mathbf p}$ an absolutely continuous stationary (acs) measure for $F$.

\medskip
We set $\alpha_{\min} = \min\{\alpha_j: j \in \Sigma\}$ and throughout the article we assume $\alpha_{\min} < 1$. Furthermore, we set
\begin{align*}
& \eta = \sum_{r \in \Sigma_R} p_r K_r^{ -\alpha_{\min}},\\
& \gamma = \sup\{\delta \geq 0:  \sum_{r \in \Sigma_R} p_r K_r^{-\delta} < 1 \}.
\end{align*}
Note that if $\eta < 1$, then $\gamma > \alpha_{\min}$. We have the following main results.

\begin{theorem}\label{thrm1.1a}
Suppose $\eta > 1$. Then $F$ admits no acs probability measure.
\end{theorem}

\begin{theorem}\label{thrm1.1b}
Suppose $\eta < 1$.
\begin{enumerate}[(1)]
\item There exists a unique acs probability measure $\mu_{\mathbf p}$ for $F$. Moreover, $F$ is ergodic with respect to $\mathbb{P} \times \mu_{\mathbf p}$.
\item The density $\frac{d\mu_{\mathbf p}}{d\lambda}$ is bounded away from zero and on the intervals $(0,\frac{1}{2}]$ and $(\frac{1}{2},1]$ is decreasing and locally Lipschitz. Furthermore, for each $\beta \in (\alpha_{\min},\gamma) \cap (0,1]$ there exist $a_1, a_2 > 0$ such that
\begin{align}
&\frac{d\mu_{\mathbf p}}{d\lambda}(x) \leq a_1 \cdot x^{-\alpha_{\min}-1+\beta}, & x \in \Big(0,\frac{1}{2}\Big], \label{eq1.5b}\\
& \frac{d\mu_{\mathbf p}}{d\lambda}(x) \leq a_2 \cdot \Big(x-\frac{1}{2}\Big)^{-1+\beta}, & x \in \Big(\frac{1}{2},1\Big]. \label{eq1.6b}
\end{align} 
\end{enumerate}
\end{theorem}


The previous theorem shows that the random system undergoes a phase transition with threshold $\eta = 1$. The system admits a finite acs measure if $\eta < 1$ and if an acs measure exists in the case that $\eta > 1$ then this measure must be infinite. Note that if $\sum_{r \in \Sigma_R} p_r K_r^{-1} < 1$,  then $\gamma > 1$. So in this case we can take $\beta = 1$, and then the previous theorem says that there exists $a > 0$ such that
\begin{align}
\frac{d\mu_{\mathbf p}}{d\lambda}(x) \leq a \cdot x^{-\alpha_{\min}}, \qquad  x \in (0,1]. \label{eq1.7b}
\end{align}
This bound is also found in \cite{LSV} where only one LSV map $T_1 \in \mathfrak{S}$ with $\alpha_1 \in (0,1)$ is considered and no maps in $\mathfrak{R}$. This suggest that in case $\sum_{r \in \Sigma_R} p_r K_r^{-1} < 1$ the attraction by the maps  $\{T_j\}_{j \in \Sigma_R}$ to $\frac{1}{2}$ does not change the order of the pole of the invariant density at zero. Note however that the density in the setting of \cite{LSV} is shown to be continuous on $(0,1)$, which in general is not the case for the density in the setting of Theorem \ref{thrm1.1b}. See Figure \ref{fig2}.

\begin{figure}[h] 
\centering
\subfigure[]
{
\includegraphics[width=0.47\textwidth]{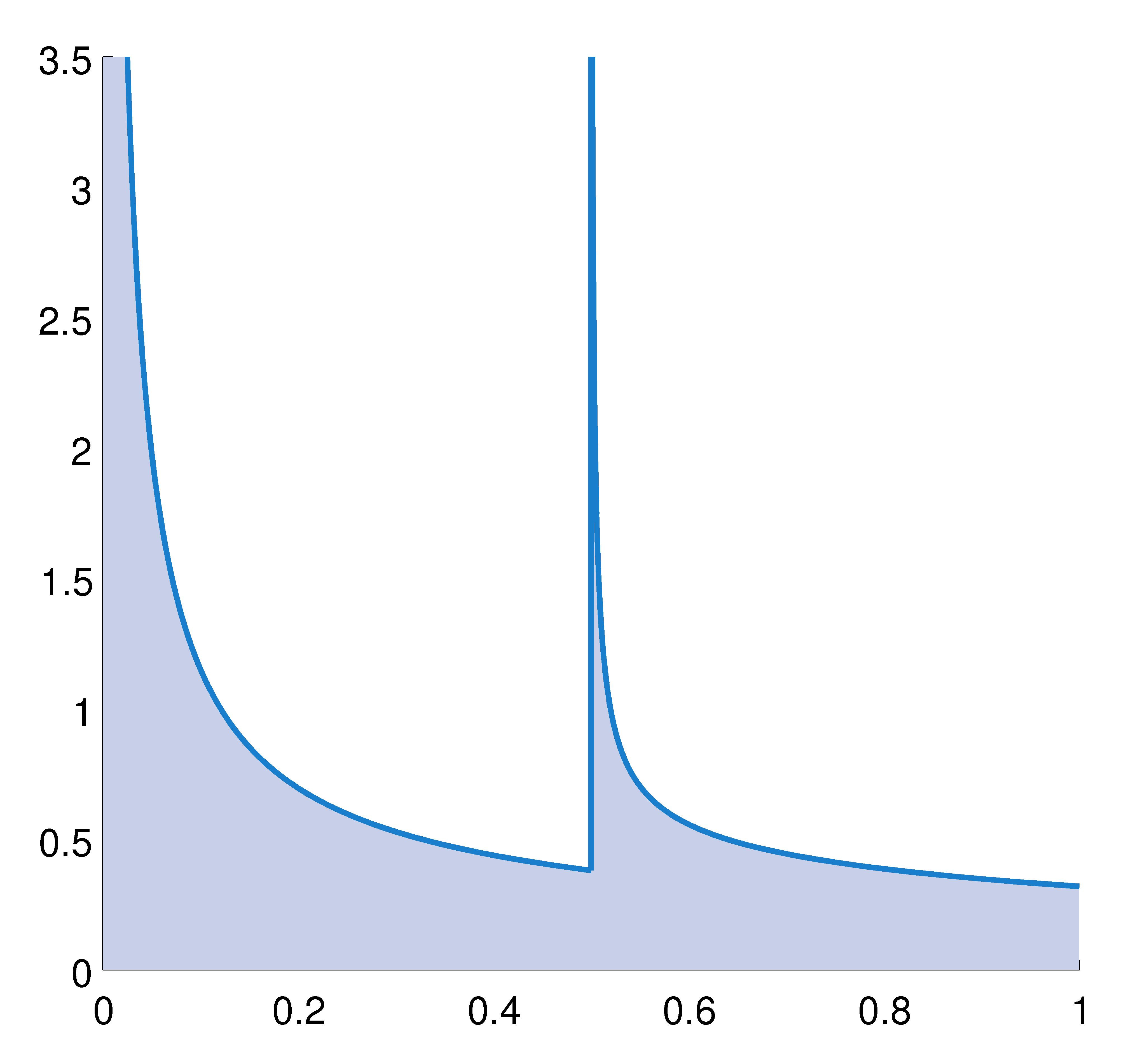}
}
\subfigure[]
{
\includegraphics[width=0.47\textwidth]{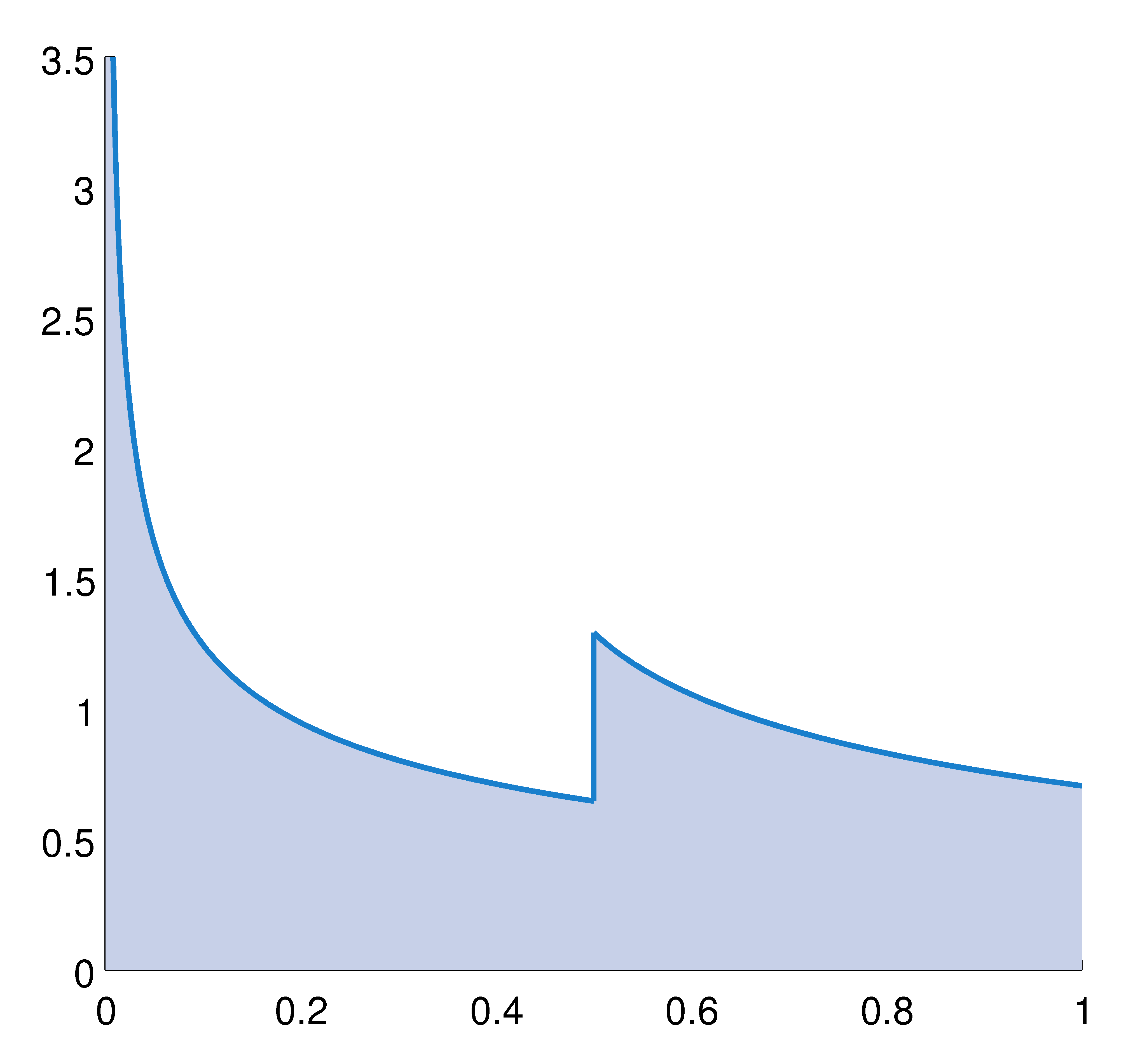}
}
\caption{Simulation of $\frac{d\mu_{\mathbf p}}{d\lambda}$ in case  $\Sigma_S=\{1\}$, $\Sigma_R = \{2\}$, $p_1 = 0.6$ and $\alpha_1 = \alpha_2 = \frac{1}{2}$ for two different values of $K_2$. Both pictures depict $P^{100}(1)$ with $P$ as in \eqref{eq5}, where in (a) we have taken $K_2 = 0.2$ (so $\eta < 1 < p_2K_2^{-1}$) and in (b) $K_2 = 0.8$ (so $\eta < p_2K_2^{-1} < 1$).}
\label{fig2}
\end{figure}

\medskip
With Theorem \ref{thrm1.1b} we can derive the following result, which says that the density $\frac{d\mu_{\mathbf p}}{d\lambda}$ depends continuously w.r.t.~the $L^1(\lambda)$-norm on the probability vector $\mathbf p \in \mathbb{R}^N$.

\begin{cor}\label{cor1.1}
For each $n \in \mathbb{N}$, let $\mathbf p_n = (p_{n,j})_{j \in \Sigma}$ be a positive probability vector such that $\sup_n \sum_{r \in \Sigma_R} p_{n,r} K_r^{-\alpha_{\min}} < 1$ and assume that $\lim_{n \rightarrow \infty} \mathbf p_n = \mathbf p$ in $\mathbb{R}_+^N$. Then $\frac{d\mu_{\mathbf p_n}}{d\lambda}$ converges with respect to the $L^1(\lambda)$-norm to $\frac{d\mu_{\mathbf p}}{d\lambda}$.
\end{cor}

\medskip
The remainder of this article is organised as follows. In Section 2 we introduce some notation and list some general preliminaries. Section \ref{sec3} concentrates on proving Theorems \ref{thrm1.1a} and \ref{thrm1.1b} and Corollary \ref{cor1.1}. First of all, Theorem \ref{thrm1.1a}, which states that $F$ admits no acs probability measure if $\eta > 1$, is proved using Kac's Lemma. Then for the case that $\eta < 1$ we show the existence of an acs probability measure by considering a suitable set of functions that is invariant with respect to the Perron-Frobenius operator of the random system. We will then apply the Arzelà-Ascoli Theorem to prove that this set has a fixed point. This approach is similar to the one in Section 2 of \cite{LSV} where only one LSV map is considered. Section \ref{sec3} ends with the proof of Corollary \ref{cor1.1} and the article will be concluded in Section \ref{sec4} with some final remarks.

\section{Preliminaries}\label{sec:mr}

In this section we introduce some notation and state some general preliminaries.

\medskip
For any finite subset $\Sigma \subseteq \mathbb N$ and any integer $n \ge 1$ we use $\mathbf u \in \Sigma^n$ to denote a {\em word} $\mathbf u = u_1 \cdots u_n$. $\Sigma^0$ contains only the empty word, which we denote by $\epsilon$. On the space of infinite sequences $\Sigma^\mathbb N$ we use
\[ [\mathbf u] = [u_1 \cdots u_n] = \{\omega \in \Sigma^{\mathbb{N}}: \omega_1 = u_1, \ldots, \omega_n = u_n\}\]
to denote the {\em cylinder set} corresponding to $\mathbf u$.  For two words $\mathbf u \in \Sigma^n$ and $\mathbf v \in \Sigma^m$ the concatenation of $\mathbf u$ and $\mathbf v$ is denoted by $\mathbf{uv} \in \Sigma^{n+m}$.

\medskip
Let $\{ T_j: [0,1] \to [0,1]\}_{j \in \Sigma}$ be a finite family of Borel measurable maps, and let $F$ be the skew product on $\Sigma^{\mathbb{N}} \times [0,1]$ given by
\begin{align*}
F(\omega,x) = (\sigma \omega, T_{\omega_1}(x)).
\end{align*}
We use the following notation for compositions of $T_1,\ldots,T_N$. For each $\omega \in \Sigma^{\mathbb{N}}$ and each $n \in \mathbb{N}$ we write
\begin{align*}
T_\omega^n(x) = T_{\omega_n}\circ  T_{\omega_{n-1}}\circ \cdots \circ T_{\omega_1}(x).
\end{align*}
Using this, we can write iterates of $F$ as
\begin{align*}
F^n(\omega,x) = (\sigma^n \omega, T_{\omega}^n(x)).
\end{align*}

\medskip
We have the following lemma on invariant measures for $F$.

\begin{lemma}[\cite{Mor85}, see also Lemma 3.2 of \cite{Fro99}]\label{l:productmeasure}
If all maps $T_j$ are non-singular with respect to $\lambda$ (that is, $\lambda(A) =0$ if and only if $\lambda(T_j^{-1}A) = 0$ holds for all $A \subseteq [0,1]$ Borel measurable) and $\mathbb P$ is the $\mathbf p$-Bernoulli measure on $\Sigma^{\mathbb{N}}$ for some positive probability vector $\mathbf p$, then the $\mathbb{P} \times \lambda$-absolutely continuous $F$-invariant probability measures are precisely the measures of the form $\mathbb{P} \times \mu$ where $\mu$ is a $\lambda$-absolutely continuous probability measure that satisfies
\begin{align}\label{eqn9}
\sum_{j \in \Sigma} p_j \mu (T_j^{-1}A) = \mu (A) \qquad \text{for all Borel sets $A$}.
\end{align}
\end{lemma}

A functional analytic approach can be used for finding measures $\mu$ that satisfy \eqref{eqn9} and are absolutely continuous w.r.t.~$\lambda$. Below we give a result for specific random interval maps on $[0,1]$. First of all, let $T: [0,1] \rightarrow [0,1]$ be piecewise strictly monotone and $C^1$. Then the {\em Perron-Frobenius operator} $\mathcal P_T: L^1(\lambda) \rightarrow L^1(\lambda)$ associated to $T$ is given by
\begin{equation}\label{q:pfd}
\mathcal P_T h (x) = \sum_{y \in T^{-1}\{x\}} \frac{h(y)}{|DT(y)|}.
\end{equation}
A non-negative function $\varphi \in L^1(\lambda)$ is a fixed point of $\mathcal P_T$ if and only if the measure $\mu$ given by $\mu(A) = \int_A \varphi \, d\lambda$ for each Borel set $A \subseteq [0,1]$ is an invariant measure for $T$. Now let $\{ T_j: [0,1] \to [0,1] \}_{j \in \Sigma}$ be a finite family of transformations such that each map $T_j$ is piecewise strictly monotone and $C^1$ and let $F$ be the corresponding skew product. Furthermore, let $\mathbf p = (p_j)_{j \in \Sigma}$ be a positive probability vector. Then the Perron-Frobenius operator $\mathcal{P}_{F,\mathbf p}: L^1(\lambda) \rightarrow L^1(\lambda)$ associated to $F$ and $\mathbf p$ is given by
\begin{align}\label{eqn3.22}
\mathcal{P}_{F,\mathbf p}h(x) = \sum_{j \in \Sigma} p_j \mathcal P_{T_j} h (x),
\end{align}
where each $\mathcal{P}_{T_j}$ is as given in \eqref{q:pfd}. A non-negative function $\varphi \in L^1(\lambda)$ is a fixed point of $\mathcal{P}_{F,\mathbf p}$ if and only if the measure $\mu$ given by $\mu(A) = \int_A \varphi d\lambda$ for each Borel set $A \subseteq [0,1]$ satisfies \eqref{eqn9}.

\medskip
Now let $(X,\mathcal{F},m)$ be a measure space and $T: X \rightarrow X$ measurable. For a set $Y \in \mathcal{F}$ the \emph{first return time map} $\varphi_Y: Y \rightarrow \mathbb{N} \cup \{\infty\}$ is defined as
\begin{align}\label{eq20}
\varphi_Y(y) = \inf\{n \geq 1 : T^n(y) \in Y\}.
\end{align}

\begin{lemma}[Kac's Lemma, see e.g.~1.5.5.~in \cite{Aar97}]\label{l:kac}
Let $T$ be an ergodic measure preserving transformation on $(X,\mathcal{F},m)$. Suppose that $m$ is finite. Let $Y \in \mathcal{F}$ be such that $m(Y) > 0$. Then $\int_Y \varphi_Y dm = m(X)$.
\end{lemma}

\section{Phase transition for the acs measure}\label{sec3}

As in the Introduction, let $T_1,\ldots,T_N \in \mathfrak{S} \cup \mathfrak{R}$ be a finite collection, write $\Sigma_S = \{1 \leq j \leq N: T_j \in \mathfrak{S}\}$, $\Sigma_R = \{1 \leq j \leq N: T_j \in \mathfrak{R}\}$ and $\Sigma = \{1,\ldots,N\} = \Sigma_S \cup \Sigma_R$ and assume that $\Sigma_S,\Sigma_R \neq \emptyset$ and $\alpha_{\min} < 1$. Furthermore, we again denote by $F$ the skew product given by \eqref{q:skewproduct}, let $\mathbf p = (p_j)_{j \in \Sigma}$ be a probability vector with strictly positive entries and let $\mathbb P$ be the $\mathbf p$-Bernoulli measure on $\Sigma^\mathbb N$. Also, recall that
\begin{align*}
\eta = \sum_{r \in \Sigma_R} p_r K_r^{-\alpha_{\min}}.
\end{align*}

\subsection{The case $\eta > 1$}\label{case1}

In this subsection we prove Theorem \ref{thrm1.1a}, namely that any acs measure for $F$ must be infinite if $\eta > 1$. For this we will use the following well-known results.

\medskip
Let $j \in \Sigma$ and define the sequence $\{x_n(j)\}$ in $(0,\frac{1}{2}]$ by
\begin{align*}
x_1(j) = \frac{1}{2} \quad \text{ and } \quad x_n(j) = T_j|_{[0,\frac{1}{2}]}^{-1}\big(x_{n-1}(j)\big) \text{ for each integer $n \geq 2$}.
\end{align*}
As explained in e.g.~the beginning of Section 6.2 of \cite{Y99} there exists a constant $C_j > 1$ such that for each $n \in \mathbb{N}$
\begin{align}\label{eq5a}
C_j^{-1} n^{-\frac{1}{\alpha_j}} \leq x_n(j) \leq C_j n^{-\frac{1}{\alpha_j}}.
\end{align}
Furthermore, we define for each $\omega \in \Sigma^{\mathbb{N}}$ the random sequence $\{x_n(\omega)\}$ in $(0,\frac{1}{2}]$ by
\begin{align*}
x_1(\omega) = \frac{1}{2} \quad \text{ and } \quad x_n(\omega) = T_{\omega_1}|_{[0,\frac{1}{2}]}^{-1}(x_{n-1}(\sigma \omega)) \text{ for each integer $n \geq 2$}.
\end{align*}
Then, for each $\omega \in \Sigma^{\mathbb{N}}$ and $n \in \mathbb{N}$,
\begin{align}\label{eq3.5g}
T_{\omega}^{n-1}((x_{n+1}(\omega),x_n(\omega)]) = \Big(x_2(\sigma^{n-1}\omega),\frac{1}{2}\Big].
\end{align}
Letting $i \in \Sigma$ be such that $\alpha_i = \alpha_{\min}$, it has been shown in Lemma 4.4 of \cite{BBD14} that for each $\omega \in \Sigma^{\mathbb{N}}$ and $n \in \mathbb{N}$ we have
\begin{align}\label{eq3.6g}
x_n(i) \leq x_n(\omega).
\end{align}


\begin{proof}[Proof of Theorem \ref{thrm1.1a}]
Suppose that $\eta > 1$ and that $\mu$ is an acs probability measure for $F$. We will use Kac's Lemma to arrive at a contradiction. Define
\begin{align*}
& A_j = \Big(x_2(j),T_j|_{[0,\frac{1}{2}]}^{-1}\Big(\frac{3}{4}\Big)\Big), \qquad j \in \Sigma, \\
& B_j = \Big(\frac{3}{4},T_j|_{(\frac{1}{2},1]}^{-1}\Big(\frac{3}{4}\Big)\Big), \qquad j \in \Sigma, \\
& Y = \bigcup_{j \in \Sigma} [j] \times (A_j \cup B_j).
\end{align*}
We consider the first return time map $\varphi_Y$ to $Y$ under $F$ as defined in \eqref{eq20}. Since $\eta > 1$, there exists $\delta > 0$ small enough such that
\begin{align*}
\gamma :=\sum_{r \in \Sigma_R} p_r L_r^{-\alpha_{\min}} > 1, \qquad \text{ where $L_r := K_r+2(1-K_r) \cdot \delta$ for each $r \in \Sigma_R$.}
\end{align*}
For each $x \in (\frac{1}{2},\frac{1}{2}+\delta)$ we have
\begin{align}\label{eq3.7z}
T_r(x) = \frac{1}{2} + \Big(K_r + 2(1-K_r)\Big(x-\frac{1}{2}\Big)\Big)\Big(x-\frac{1}{2}\Big) \leq \frac{1}{2} + L_r \Big(x-\frac{1}{2}\Big).
\end{align}
For $\mathbf r = (r_1,\ldots,r_n) \in \Sigma_R^n$ we write $L_{\mathbf r} = \prod_{l=1}^n L_{r_l}$ with $L_{\mathbf r} = 1$ if $n=0$. Furthermore, fix $t \in \Sigma_R$. It is easy to see that $\lim_{n \rightarrow \infty} T_t^n(\frac{3}{4}) = \frac{1}{2}$, so there exists an integer $k \geq 0$ such that $T_t^k(\frac{3}{4}) \in (\frac{1}{2},\frac{1}{2} + \delta)$ holds.

\medskip
Let $(\omega,x) \in Y$ and $t$ and $k$ be as above. Furthermore, fix $s \in \Sigma_S$. Suppose that
$$\omega \in [u\underbrace{t\cdots t}_{k\ \text{times}} \mathbf r s] =  [ut^k \mathbf r s], \quad \text{for some $u \in \Sigma$, $\mathbf r \in \Sigma_R^n, n \geq 0$.}$$
We then have $T_{\omega}^l(x) \in (\frac{1}{2},\frac{3}{4})$ for all $1 \leq l \leq 1+k+n$. It follows from $T_{\omega_1}(x) \leq \frac{3}{4}$, $T_t^k(\frac{3}{4}) \in (\frac{1}{2},\frac{1}{2}+\delta)$ and \eqref{eq3.7z} that
\begin{align*}
T_{\omega}^{1+k+n}(x) \leq T_{\sigma \omega}^{k+n}\Big(\frac{3}{4}\Big) \leq \frac{1}{2}+L_{\mathbf r}\Big(T_t^k\Big(\frac{3}{4}\Big)-\frac{1}{2}\Big),
\end{align*}
which gives
\begin{align}\label{eq3.8c}
T_{\omega}^{2+k+n}(x) \leq L_{\mathbf r}\Big(2T_t^{k}\Big(\frac{3}{4}\Big)-1\Big)
\end{align}
Fix $i \in \Sigma$ such that $\alpha_i = \alpha_{\min}$. There exists an $m \in \mathbb{N}$ such that $T_{\omega}^{2+k + n}(x) \in (x_{m+1}(i),x_m(i)]$. It follows from \eqref{eq3.5g} and \eqref{eq3.6g} that
\begin{align}\label{eq3.9x}
\varphi_Y(\omega,x) \geq 2+ k + n+m.
\end{align}
We give a lower bound for $m$ in terms of $\mathbf r$. It follows from \eqref{eq5a} and \eqref{eq3.8c} that
\begin{align*}
C_i^{-1} (m+1)^{-\frac{1}{\alpha_i}} \leq L_{\mathbf r}\Big(2T_t^{k}\Big(\frac{3}{4}\Big)-1\Big).
\end{align*}
Solving for $m$ yields
\begin{align}\label{eq3.11x}
m \geq M_1\cdot  L_{\mathbf r}^{-\alpha_i}-1,
\end{align}
where we defined $M_1 = C_i^{-\alpha_i} \cdot (2T_t^{k}(\frac{3}{4})-1)^{-\alpha_i}$. Combining \eqref{eq3.9x} and \eqref{eq3.11x} yields
\begin{equation}\begin{split}
\int_Y \varphi_Y d\mathbb{P} \times \mu & \geq \sum_{u \in \Sigma} \sum_{n=0}^{\infty} \sum_{\mathbf r \in \Sigma_R^n} \int_{[u t^{k} \mathbf r s] \times (A_u \cup B_u)} \varphi_Y d\mathbb{P} \times \mu \\
& \geq \sum_{u \in \Sigma} \sum_{n=0}^{\infty} \sum_{\mathbf r \in \Sigma_R^n} \mathbb{P}\big([u t^{k} \mathbf r s]\big) \int_{A_u \cup B_u}  M_1 \cdot  L_{\mathbf r}^{-\alpha_i} d \mu(x) \\
& = M_2 \cdot \sum_{n=0}^{\infty} \gamma^n, \label{eq3.12d}
\end{split}\end{equation}
where
\begin{align*}
M_2 = M_1 \cdot p_t^k p_s \cdot \sum_{u \in \Sigma} p_u \mu(A_u \cup B_u) = M_1 \cdot p_t^k p_s \cdot \mathbb{P} \times \mu(Y).
\end{align*}
Almost every orbit that starts in $\Sigma^{\mathbb{N}} \times (\frac{1}{2},\frac{3}{4})$ will eventually enter $\Sigma^{\mathbb{N}} \times (\frac{1}{2},1)$ under applications of $F$. Conversely, almost every orbit that starts in $\Sigma^{\mathbb{N}} \times (\frac{1}{2},1)$ will eventually enter $\Sigma^{\mathbb{N}} \times (\frac{1}{2},\frac{3}{4})$, either via $\bigcup_{j \in \Sigma} [j] \times A_j$ or via $\bigcup_{j \in \Sigma} [j] \times B_j$. Hence, we have $\bigcup_{n=0}^{\infty} F^{-n} Y = \Sigma^{\mathbb{N}} \times [0,1]$ up to some set of measure zero. This together with the $F$-invariance of $\mathbb{P} \times \mu$ yields
\begin{align*}
1 = \mathbb{P} \times \mu(\Sigma^{\mathbb{N}} \times [0,1]) \leq \sum_{n=0}^{\infty} \mathbb{P} \times \mu(F^{-n} Y) = \sum_{n=0}^{\infty} \mathbb{P} \times \mu(Y).
\end{align*}
This gives $\mathbb{P} \times \mu(Y) > 0$ and so $M_2 > 0$. Hence, from \eqref{eq3.12d} and $\gamma \geq 1$ it now follows that
\begin{equation}\label{eq3.13s}
\int_Y \varphi_Y d\mathbb{P} \times \mu = \infty.
\end{equation}
On the other hand, since $\mu$ is a probability measure by assumption, we obtain from the Ergodic Decomposition Theorem, see e.g.~\cite[Theorem 6.2]{einsiedler11}, that there exists a probability space $(E,\mathcal{E},\nu)$ and a measurable map $e \mapsto \mu_e$ with $\mu_e$ an $F$-invariant ergodic probability measure for $\nu$-a.e.~$e \in E$, such that
\begin{align*}
\int_Y \varphi_Y d\mathbb{P} \times \mu = \int_E \Big(\int_Y \varphi_Y d\mu_e\Big) d\nu(e).
\end{align*}
For each $e \in E$ for which $\mu_e$ is an $F$-invariant ergodic probability measure we have $\int_Y \varphi_Y d\mu_e = \mu_e(X) = 1$ if $\mu_e(Y) > 0$ by Lemma \ref{l:kac} and we have $\int_Y \varphi_Y d\mu_e = 0$ if $\mu_e(Y) = 0$. This gives
\begin{align*}
\int_Y \varphi_Y d\mathbb{P} \times \mu \leq \nu(E) = 1,
\end{align*}
which is in contradiction with \eqref{eq3.13s}.
\end{proof}

\subsection{The case $\eta < 1$}\label{subsec3.2} In this subsection we will prove Theorem \ref{thrm1.1b} and Corollary \ref{cor1.1}. For this we wil identify a suitable set of functions which is preserved by the Perron-Frobenius operator $P = P_{F,\mathbf p}$ associated to $F$ and $\mathbf p$ as given in \eqref{eqn3.22}. We will do this in a number of steps in a way that is similar to the approach of Section 2 in \cite{LSV}.

\medskip
Suppose $\eta < 1$. On $[0,1]$ we define for each $j \in \Sigma$ the functions $x \mapsto y_j(x)$ and $x \mapsto \xi_j(x)$ by $y_j(x) = (T_j|_{[0,\frac{1}{2}]})^{-1} (x)$ and $\xi_j(x) = (2y_j)^{\alpha_j}$. Furthermore, we define on $[0,1]$ the function $z(x) = \frac{x+1}{2}$ and on $(\frac{1}{2},1]$ we define for each $r \in \Sigma_R$ the function $z_r(x)=(T_r|_{(\frac{1}{2},1]})^{-1} (x)$. Whenever convenient, we will just write $y_j$ for $y_j(x)$ and similarly for $\xi_j$, $z$ and $z_r$. Writing $p_S = \sum_{s \in \Sigma_S} p_s$, we then have 
\begin{align}\label{eq5}
P f(x) = \begin{cases}
\sum_{j \in \Sigma} p_j \frac{f(y_j)}{1+(\alpha_j+1)\xi_j} + p_S \frac{f(z)}{2}, & x \in [0,\frac{1}{2}]\\
\sum_{j \in \Sigma} p_j \frac{f(y_j)}{1+(\alpha_j+1)\xi_j} + p_S \frac{f(z)}{2} + \sum_{r \in \Sigma_R} p_r \frac{f(z_r)}{DR_{\alpha,K}(z_r)}, & x \in (\frac{1}{2},1].
\end{cases}
\end{align}
Note that $x \mapsto y_j(x)$, $x \mapsto \xi_j(x)$, $x \mapsto z(x)$ and $x \mapsto z_r(x)$ are increasing and continuous on $(0,\frac{1}{2}]$ and $(\frac{1}{2},1]$. This in combination with the fact that $R_{\alpha,K}$ is $C^1$ on $(\frac{1}{2},1]$ with increasing derivative gives that the set
$$\mathcal{C}_0 = \Big\{f \in L^1(\lambda): f \geq 0, \text{$f$ decreasing and continuous on $\Big(0,\frac{1}{2}\Big]$ and $\Big(\frac{1}{2},1\Big]$}\Big\}$$
is preserved by $P$, i.e.~$P \mathcal{C}_0 \subseteq \mathcal{C}_0$.

\medskip
Since $\eta < 1$, we have $\gamma = \sup\{\delta \geq 0:  \sum_{r \in \Sigma_R} p_r K_r^{-\delta} < 1 \} > \alpha_{\min}$, so $(\alpha_{\min},\gamma)$ is non-empty. In the remainder of this subsection we fix a $\beta \in (\alpha_{\min},\gamma) \cap (0,1]$. We set $\alpha_{\max} = \max\{\alpha_j: j \in \Sigma\}$ and $d = \alpha_{\max}+2$. We need the following two lemma's.

\begin{lemma}\label{lemma3.2}
For each $\alpha > 0$ the function $x \mapsto \frac{(1+x)^{d}}{1+(\alpha+1)x}$ is increasing on $[0,1]$.
\end{lemma}
\begin{proof}
Set
$$f_{\alpha}(x) = \frac{(1+x)^{d}}{1+(\alpha+1)x}, \qquad x \in [0,1].$$
Furthermore, set $g(x) = (1+x)^{d}$ and $h_{\alpha}(x) = 1+(\alpha+1)x$ where $x \in [0,1]$. Then
\begin{align*}
f_{\alpha}'(x) = \frac{h_{\alpha}(x) g'(x) - g(x) h_{\alpha}'(x)}{h_{\alpha}(x)^2}.
\end{align*}
We have
\begin{align*}
h_{\alpha}(x)g'(x) &= (1+(\alpha+1)x) \cdot d(1+x)^{d-1} \\  & \geq (1+x)^{d} \cdot d
 \geq  (1+x)^{d} \cdot (\alpha+1) \\
&= g(x) h_{\alpha}'(x),
\end{align*}
so $f_{\alpha}'(x) \geq 0$ holds for all $x \in [0,1]$.
\end{proof}

Define for each $K > 0$ and $b \geq 0$ the function $H_{K,b}: [\frac{1}{2},1] \rightarrow \mathbb{R}$ by
$$H_{K,b}(x) = \frac{(K+2(1-K)(x-\frac{1}{2}))^{b}}{K+4(1-K)(x-\frac{1}{2})}, \qquad x \in \Big[\frac{1}{2},1\Big].$$
\begin{lemma}\label{lemma3.2b}
Let $K > 0$ and $b \geq 0$.
\begin{enumerate}[(i)]
\item If $b \geq 2$, then $H_{K,b}$ is increasing.
\item If $b \leq 1$, then $H_{K,b}$ is decreasing.
\end{enumerate}
\end{lemma}
\begin{proof}
Set $f_K(x) =K+2(1-K)(x-\frac{1}{2})$ and $g_K(x) = K+4(1-K)(x-\frac{1}{2})$ where $x \in [\frac{1}{2},1]$. Note that $g_K'(x) = 2f_K'(x)$. Then for $x \in (\frac{1}{2},1)$
\begin{align*}
H_{K,b}'(x) &= \frac{g_K(x) \cdot b \cdot f_K(x)^{b-1}f_K'(x) - f_K(x)^b \cdot g_{K}'(x)}{g_{K}(x)^2}\\
&= \frac{f_K(x)^b \cdot f_K'(x) \big(b \cdot \frac{g_K(x)}{f_K(x)}-2\big)}{g_K(x)^2}.
\end{align*}
If $b \geq 2$, then
\begin{align*}
b \cdot \frac{g_K(x)}{f_K(x)}-2 \geq 2 \cdot \frac{g_K(x)}{f_K(x)}-2 \geq 2 \cdot \frac{f_K(x)}{f_K(x)}-2 = 0
\end{align*}
and thus $H_{K,b}'(x) \geq 0$. This proves (i). If $b \leq 1$, then
\begin{align*}
b \cdot \frac{g_K(x)}{f_K(x)}-2 \leq \frac{g_K(x)}{f_K(x)}-2 \leq \frac{2 f_K(x)}{f_K(x)}-2 = 0
\end{align*}
and thus $H_{K,b}'(x) \leq 0$. This proves (ii). 
\end{proof}

We can now prove the following lemma.

\begin{lemma}\label{lemma3.3}
The set
$$\mathcal{C}_1 = \Big\{f \in \mathcal{C}_0: x \mapsto x^{d} f(x) \text{ incr.~on $\Big(0,\frac{1}{2}\Big]$, } x \mapsto \Big(x-\frac{1}{2}\Big)^{d} f(x) \text{ incr.~on $\Big(\frac{1}{2},1\Big]$}\Big\}$$
is preserved by $P$.
\end{lemma}
\begin{proof}
Let $f \in \mathcal{C}_1$. Let $x \in (0,\frac{1}{2}]$. Using that for each $j \in \Sigma$ we have $x = y_j(1+\xi_j)$ and that $z(x)-\frac{1}{2} = \frac{x}{2}$, we obtain
\begin{align*}
x^{d} Pf(x) &= \sum_{j \in \Sigma} p_j \Big(\frac{x}{y_j}\Big)^{d} \frac{y_j^{d} f(y_j)}{1+(\alpha_j+1)\xi_j} + \frac{p_S}{2} \Big(\frac{x}{z-\frac{1}{2}}\Big)^{d} \Big(z - \frac{1}{2}\Big)^{d} f(z) \\
&= \sum_{j \in \Sigma} p_j \frac{(1+\xi_j)^{d}}{1+(\alpha_j+1)\xi_j} \cdot y_j^{d} f(y_j) + p_S \cdot 2^{d} \cdot \Big(z - \frac{1}{2}\Big)^{d} f(z).
\end{align*}
Because $x \mapsto \xi_j(x)$ is increasing for each $j \in \Sigma$ it follows from Lemma \ref{lemma3.2} that $x \mapsto \frac{(1+\xi_j(x))^{d}}{1+(\alpha_j+1)\xi_j(x)}$ is increasing for each $j \in \Sigma$. Combining this with the fact that $f \in \mathcal{C}_1$, that $y_j \in (0,\frac{1}{2}]$ for each $j \in \Sigma$ and that $z \in (\frac{1}{2},1]$ we conclude that $x \mapsto x^{d} Pf(x)$ is increasing on $(0,\frac{1}{2}]$.

\medskip
Now let $x \in (\frac{1}{2},1]$. Then
\begin{align*}
\Big(x-\frac{1}{2}\Big)^{d} Pf(x) =&  \Big(\frac{x-\frac{1}{2}}{x}\Big)^{d} \sum_{j \in \Sigma} p_j \Big(\frac{x}{y_j}\Big)^{d} \frac{y_j^{d} f(y_j)}{1+(\alpha_j+1)\xi_j} \\
& + \frac{p_S}{2} \Big(\frac{x-\frac{1}{2}}{z-\frac{1}{2}}\Big)^{d} \Big(z - \frac{1}{2}\Big)^{d} f(z) \\
& + \sum_{r \in \Sigma_R} \frac{p_r}{DR_{\alpha_r,K_r}(z_r)} \Big(\frac{x-\frac{1}{2}}{z_r-\frac{1}{2}}\Big)^{d} \Big(z_r - \frac{1}{2}\Big)^{d} f(z_r).
\end{align*}
Using again that for each $j \in \Sigma$ we have $x = y_j(1+\xi_j)$, that $z-\frac{1}{2} = \frac{x}{2}$ and also that $x- \frac{1}{2} = K_r(z_r - \frac{1}{2})+2(1-K_r)(z_r-\frac{1}{2})^2$ for each $r \in \Sigma_R$, we obtain
\begin{align*}
\Big(x-\frac{1}{2}\Big)^{d} Pf(x) =&  \Big(1-\frac{1}{2x}\Big)^{d} \sum_{j \in \Sigma} p_j \frac{(1+\xi_j)^{d}}{1+(\alpha_j+1)\xi_j} \cdot y_j^{d} f(y_j) \\
& + \frac{p_S}{2} \Big(2-\frac{1}{x}\Big)^{d} \Big(z - \frac{1}{2}\Big)^{d} f(z) \\
& + \sum_{r \in \Sigma_R} p_r \frac{(K_r+2(1-K_r)(z_r-\frac{1}{2}))^{d}}{K_r+4(1-K_r)(z_r-\frac{1}{2})} \Big(z_r - \frac{1}{2}\Big)^{d} f(z_r).
\end{align*}
Note that $x \mapsto (1-\frac{1}{2x})^{d}$ and $x \mapsto (2-\frac{1}{x})^{d}$ are positive and increasing on $(\frac{1}{2},1]$. Combining this with Lemma \ref{lemma3.2} and Lemma \ref{lemma3.2b}(i) and with the fact that $f \in \mathcal{C}_1$ we conclude that $x \mapsto (x- \frac{1}{2})^{d} Pf(x)$ is increasing on $(\frac{1}{2},1]$.
\end{proof}

We set $t_1 = \alpha_{\min} + 1 - \beta$ and $t_2 = 1-\beta$. It follows from $\beta \in (\alpha_{\min},1]$ that $t_1 \in [\alpha_{\min},1)$ and $t_2 \in [0,1-\alpha_{\min})$.

\begin{lemma}\label{lemma3.4h}
For sufficiently large $a_1,a_2 > 0$, the set
$$\mathcal{C}_2 = \Big\{ f \in \mathcal{C}_1: f(x) \leq a_1 x^{-t_1} \text{ on $\Big(0,\frac{1}{2}\Big]$, } f(x) \leq a_2 \Big(x-\frac{1}{2}\Big)^{-t_2} \text{ on $\Big(\frac{1}{2},1\Big]$}, \int_0^1 f d\lambda = 1 \Big\}$$
is preserved by $P$.
\end{lemma}

\begin{proof}
Let $f \in \mathcal{C}_2$. First, let $x \in (\frac{1}{2},1]$. For each $j \in \Sigma$ we have $y_j \leq \frac{1}{2}$ and thus, using that $f \in \mathcal{C}_1$,
\begin{align*}
y_j^{d} f(y_j) \leq 2^{-d} f\Big(\frac{1}{2}\Big) \leq 2^{-d} \cdot 2 \cdot \int_0^{\frac{1}{2}}f(u) du \leq 2^{-d+1}.
\end{align*}
Furthermore, for each $j \in \Sigma$ we have
\begin{align*}
T_j\Big(\frac{1}{4}\Big) = \frac{1}{4}(1+2^{-\alpha_j}) \leq \frac{1}{4}(1+1) = \frac{1}{2},
\end{align*}
which gives $y_j \in (\frac{1}{4},\frac{1}{2}]$. Setting $M := 2^{d+1}$ we obtain for each $j \in \Sigma$ that
\begin{align}\label{eq6}
\frac{f(y_j)}{1+(\alpha_j+1) \xi_j} = y_j^{d} f(y_j) \cdot \frac{y_j^{-d}}{1+(\alpha_j+1) \cdot (2y_j)^{\alpha_j}} \leq 2^{-d+1} \cdot 4^{d} = M.
\end{align}
It also follows from $f \in \mathcal{C}_1$ that
\begin{align*}
\Big(z-\frac{1}{2}\Big)^{d} f(z) \leq \Big(1-\frac{1}{2}\Big)^{d} f(1) \leq 2^{-d} \cdot 2 \cdot \int_{\frac{1}{2}}^1 f(u) du \leq 2^{-d+1}.
\end{align*}
Using that $z \in (\frac{3}{4},1]$, this gives
\begin{align}\label{eq7}
f(z) \leq 2^{-d+1} \cdot \Big(z - \frac{1}{2}\Big)^{-d} \leq 2^{-d+1} \cdot \Big(\frac{3}{4}-\frac{1}{2}\Big)^{-d} = M.
\end{align}
Combining \eqref{eq5}, \eqref{eq6} and \eqref{eq7} and using that $f \in \mathcal{C}_2$ gives
\begin{align}\label{eq7a}
Pf(x) & \leq M + \frac{p_S}{2} \cdot M + \sum_{r \in \Sigma_R} \frac{p_r}{DR_{\alpha_r,K_r}(z_r)} \cdot a_2 \Big(z_r-\frac{1}{2}\Big)^{-t_2}.
\end{align}
For each $r \in \Sigma_R$ we have $x-\frac{1}{2} = K_r(z_r-\frac{1}{2})+2(1-K_r)(z_r-\frac{1}{2})^2$ and therefore
\begin{align*}
\frac{1}{DR_{\alpha_r,K_r}(z_r)}\Big(\frac{x-\frac{1}{2}}{z_r-\frac{1}{2}}\Big)^{t_2} = \frac{(K_r+2(1-K_r)(z_r-\frac{1}{2}))^{t_2}}{K_r+4(1-K_r)(z_r-\frac{1}{2})},
\end{align*}
which by Lemma \ref{lemma3.2b}(ii) can be bounded from above by $H_{K_r,t_2}(\frac{1}{2}) = K_r^{t_2-1}$. Furthermore, since $t_2 \geq 0$ we have $(x-\frac{1}{2})^{t_2} \leq 2^{-t_2}$. We obtain
\begin{align}\label{eq8}
Pf(x) \leq \Big\{ \frac{M(1 + \frac{p_S}{2}) \cdot 2^{-t_2}}{a_2} + \sum_{r \in \Sigma_R} p_r \cdot K_r^{t_2 -1} \Big\} \cdot a_2 \cdot \Big(x-\frac{1}{2}\Big)^{-t_2}.
\end{align}
We have $t_2 - 1 = -\beta$ and $\beta < \gamma$, so
\begin{align*}
\sum_{r \in \Sigma_R} p_r \cdot K_r^{t_2-1} = \sum_{r \in \Sigma_R} p_r \cdot K_r^{-\beta}  < 1.
\end{align*}
Hence, there exists an $a_2 > 0$ sufficiently large such that the term in curly brackets in \eqref{eq8} is bounded by 1.

\medskip
Now let $x \in (0,\frac{1}{2}]$. Using that $f \in \mathcal{C}_2$, it follows from \eqref{eq5} that
\begin{align}\label{eq9a}
Pf(x) \leq \sum_{j \in \Sigma} p_j \frac{a_1 \cdot y_j^{-t_1}}{1+(\alpha_j+1)\xi_j} + \frac{p_S \cdot a_2}{2} \cdot \Big(z-\frac{1}{2}\Big)^{-t_2}.
\end{align}
For each $j \in \Sigma$ we have, using that $x= y_j (1+\xi_j)$ and that $t_1 \in (0,1)$,
\begin{align}\label{eq9b}
\frac{y_j^{-t_1}}{1+(\alpha_j+1)\xi_j} = \frac{x^{-t_1} (1+\xi_j)^{t_1}}{1+(\alpha_j+1)\xi_j} \leq \frac{x^{-t_1}(1+t_1 \xi_j)}{1+(\alpha_j+1)\xi_j} \leq x^{-t_1}.
\end{align}
Fix an $i \in \Sigma$ with $\alpha_i = \alpha_{\min}$. Applying for each $j \in \Sigma \backslash \{i\}$ the bound \eqref{eq9b} to \eqref{eq9a} and using that $z- \frac{1}{2} = \frac{x}{2}$ yields
\begin{align}\label{eq9}
Pf(x) \leq \Big\{ p_i  \Big(\frac{x}{y_i}\Big)^{t_1} \cdot \frac{1}{1+(\alpha_i+1)\xi_i} + (1-p_i) + \frac{p_S \cdot a_2 \cdot 2^{t_2-1}}{a_1} \cdot x^{t_1-t_2} \Big\} \cdot a_1 \cdot x^{-t_1}.
\end{align}
It remains to find $a_1$ sufficiently large such that the term in curly brackets in \eqref{eq9} is bounded by 1. First of all, using again that $x = y_i(1+\xi_i)$ and that $t_1 \in (0,1)$ we get
\begin{align}\label{eq10}
\Big(\frac{x}{y_i}\Big)^{t_1} \cdot \frac{1}{1+(\alpha_i+1)\xi_i} = \frac{(1+\xi_i)^{t_1}}{1+(\alpha_i+1)\xi_i} \leq \frac{1+t_1 \xi_i}{1+(\alpha_i+1)\xi_i}.
\end{align}
Furthermore, we have
\begin{align}\label{eq11}
x^{t_1-t_2} = x^{\alpha_{\min}} = y_i^{\alpha_i} (1+\xi_i)^{\alpha_i} \leq y_i^{\alpha_i} \cdot 2^{\alpha_i} = \xi_i.
\end{align}
It follows from \eqref{eq10} and \eqref{eq11} that the term in curly brackets in \eqref{eq9} is bounded by
\begin{align}\label{eq12}
p_i \frac{1+t_1 \xi_i + \frac{p_i^{-1} \cdot p_S \cdot a_2 \cdot 2^{t_2-1}}{a_1} \cdot \xi_i \cdot (1+(\alpha_i+1)\xi_i)}{1+(\alpha_i+1)\xi_i} + (1-p_i).
\end{align}
Using that $1+(\alpha_i+1)\xi_i \leq \alpha_i + 2$ we get that the numerator in \eqref{eq12} is bounded by
\begin{align*}
1+\Big(t_1 + \frac{p_i^{-1} \cdot p_S \cdot a_2 \cdot 2^{t_2-1}(\alpha_i+2)}{a_1}\Big)\xi_i.
\end{align*}
Taking $a_1 > 0$ sufficiently large such that $t_1 + \frac{p_i^{-1} \cdot p_S \cdot a_2 \cdot 2^{t_2-1} (\alpha_i+2)}{a_1} \leq 1 \leq \alpha_i + 1$ now yields the result.
\end{proof}

\begin{lemma}\label{lemma3.5u}
The set $\mathcal{C}_2$ is compact with respect to the $L^1(\lambda)$-norm.
\end{lemma}

\begin{proof}
For each $f \in \mathcal{C}_2$ let $\phi_f$ denote the continuous extension of $(0,\frac{1}{2}] \ni x \mapsto x^{d} f(x)$ to $[0,\frac{1}{2}]$ and let $\psi_f$ denote the continuous extension of $(\frac{1}{2},1] \ni x \mapsto (x-\frac{1}{2})^{d} f(x)$ to $[\frac{1}{2},1]$. Furthermore, we define $\mathcal{A}_1 = \{\phi_f: f \in \mathcal{C}_2\}$ and $\mathcal{A}_2 = \{\psi_f: f \in \mathcal{C}_2\}$. For each $f \in \mathcal{C}_2$ we have, for $x, y \in [0,\frac{1}{2}]$ with $x \geq y$, that
\begin{equation}\begin{split}\label{eq3.25c}
0 \leq \phi_f(x)-\phi_f(y) & \leq f(x)(x^{d} - y^{d}) \leq a_1 x^{-t_1} \cdot d \int_y^x t^{d-1} dt \\
& \leq a_1 x^{d-1-t_1} \cdot d |x-y| \leq a_1  \cdot 2^{-d+1+t_1}\cdot d |x-y|.
\end{split}\end{equation}
and for $x, y \in [\frac{1}{2},1]$ with $x \geq y$, that
\begin{equation}\begin{split}\label{eq3.26c}
0 \leq \psi_f(x)-\psi_f(y)  & \leq f(x) \Big(\Big(x-\frac{1}{2}\Big)^{d} - \Big(y-\frac{1}{2}\Big)^{d}\Big)\\
& \leq a_2 \Big(x-\frac{1}{2}\Big)^{-t_2} \cdot d \int_y^x \Big(t-\frac{1}{2}\Big)^{d-1} dt \\
& \leq a_2  \Big(x-\frac{1}{2}\Big)^{d-1-t_2} \cdot d |x-y| \leq a_2 \cdot 2^{-d+1+t_2} \cdot d|x-y|.
\end{split}\end{equation}
Also, from the definition of $\mathcal{C}_2$ in Lemma \ref{lemma3.4h} and the fact that $d > \max\{t_1,t_2\}$ we see that  $\phi_f(0) = \psi_f(\frac{1}{2}) = 0$ holds for each $f \in \mathcal{C}_2$. It follows that $\mathcal{A}_1$ and $\mathcal{A}_2$  are uniformly bounded and equicontinuous, so from the Arzel\`a-Ascoli Theorem we obtain that $\mathcal{A}_1$ and $\mathcal{A}_2$ are compact in $C([0,\frac{1}{2}])$ and $C([\frac{1}{2},1])$, respectively, w.r.t.~the supremum norm.

\medskip
Now let $\{f_n\}$ be a sequence in $\mathcal{C}_2$. It follows from the above that $\{f_n\}$ has a subsequence $\{f_{n_k}\}$ such that $\{\phi_{f_{n_k}}\}$ converges uniformly to some $\phi^* \in C([0,\frac{1}{2}])$ and $\{\psi_{f_{n_k}}\}$ converges uniformly to some $\psi^* \in C([\frac{1}{2},1])$ (for this we take a suitable subsequence of a subsequence of $\{f_n\}$). Now define the measurable function $f^*$ on $(0,1]$ by
\begin{align*}
f^*(x) =  \begin{cases}
x^{-d} \phi^*(x) & \text{if}\quad x \in (0,\frac{1}{2}],\\
(x-\frac{1}{2})^{-d} \psi^*(x) & \text{if}\quad x \in (\frac{1}{2},1].
\end{cases}
\end{align*}
Then $f^*$ is continuous on $(0,\frac{1}{2}]$ and $(\frac{1}{2},1]$. Moreover, $\{f_{n_k}\}$ converges pointwise to $f^*$. First of all, this gives $f^* \in \mathcal{C}_1$. Secondly, this gives combined with
\begin{align*}
\sup_{k \in \mathbb{N}} f_{n_k}(x) \leq a_1 x^{-t_1} \text{ for $x \in \Big(0,\frac{1}{2}\Big]$}, \qquad \sup_{k \in \mathbb{N}} f_{n_k}(x) \leq a_2 \Big(x-\frac{1}{2}\Big)^{-t_2} \text{ for $x \in \Big(\frac{1}{2},1\Big]$}
\end{align*}
and
\begin{align*}
\int_0^{\frac{1}{2}} x^{-t_1} dx < \infty, \qquad \int_{\frac{1}{2}}^1 \Big(x-\frac{1}{2}\Big)^{-t_2} dx < \infty,
\end{align*}
that $f^*(x) \leq a_1 x^{-t_1}$ for $x \in (0,\frac{1}{2}]$ and $f^*(x) \leq a_2 (x-\frac{1}{2})^{-t_2}$ for $x \in (\frac{1}{2},1]$, and that
\begin{align*}
\lim_{k \rightarrow \infty} \|f^* - f_{n_k}\|_1 = 0 \quad \text{ and so } \quad \int_0^1 f^* d\lambda = 1
\end{align*}
using the Dominated Convergence Theorem. We conclude that $f^* \in \mathcal{C}_2$ and that $f^*$ is a limit point of $\{f_n\}$ with respect to the $L^1(\lambda)$-norm.
\end{proof}

Using the previous lemma's we are now ready to prove Theorem \ref{thrm1.1b}.

\begin{proof}[Proof of Theorem \ref{thrm1.1b}]

(1) Take $f \in \mathcal{C}_2$ and define the sequence of functions $\{f_n\}$ by $f_n = \frac{1}{n} \sum_{i=0}^{n-1} P^i f$. Using that $P$ preserves $\mathcal{C}_2$ and that the average of a finite collection of elements of $\mathcal{C}_2$ is also an element of $\mathcal{C}_2$, we obtain that $\{f_n\}$ is a sequence in $\mathcal{C}_2$. It follows from Lemma \ref{lemma3.5u} that $\{f_n\}$ has a subsequence $\{f_{n_k}\}$ that converges w.r.t.~the $L^1(\lambda)$-norm to some $f^* \in \mathcal{C}_2$. As is standard, we then obtain that $Pf^*(x) = f^*(x)$ holds for $\lambda$-a.e.~$x \in [0,1]$ by noting that
\begin{align*}
\| Pf^* - f^*\|_1 & \leq \| Pf^* - P f_{n_k} \|_1 + \| P f_{n_k} - f_{n_k} \|_1 + \| f_{n_k} - f^*\|_1 \\
& \leq 2 \| f_{n_k} - f^* \|_1 + \Big \| \frac{1}{n_k} \sum_{i=0}^{n_k-1} P^{i+1} f - \frac{1}{n_k} \sum_{i=0}^{n_k-1} P^i f \Big \|_1 \\
& \leq 2 \| f_{n_k} - f^* \|_1 + \frac{1}{n_k} \| P^{n_k}f - f \|_1 \\
& \leq 2 \| f_{n_k} - f^* \|_1 + \frac{2}{n_k} \| f \|_1 \rightarrow 0, \qquad k \rightarrow \infty.
\end{align*}
Hence, $F$ admits an acs probability measure $\mu_{\mathbf p}$ with $\frac{d\mu_{\mathbf p}}{d\lambda} \in \mathcal{C}_2$. It follows that $\frac{d\mu_{\mathbf p}}{d\lambda}$ has full support on $[0,1]$, so we obtain that $\mu_{\mathbf p}$ is the only acs probability measure once we know that $F$ is ergodic with respect to $\mathbb{P} \times \mu_{\mathbf p}$. So let $A \subseteq \Sigma^{\mathbb{N}} \times [0,1]$ be Borel measurable such that $F^{-1} A = A$. Suppose $\mathbb{P} \times \mu_{\mathbf p}(A) > 0$. The probability measure $\rho$ on $\Sigma^{\mathbb{N}} \times [0,1]$ given by
\begin{align*}
\rho(B) = \frac{\mathbb{P} \times \mu_{\mathbf p}(A \cap B)}{\mathbb{P} \times \mu_{\mathbf p}(A)}
\end{align*}
for Borel measurable sets $B \subseteq \Sigma^{\mathbb{N}} \times [0,1]$ is $F$-invariant and absolutely continuous with respect to $\mathbb{P} \times \lambda$ with density $\frac{d\rho}{d\mathbb{P} \times \lambda}(\omega,x) = \frac{1}{\mathbb{P} \times \mu_{\mathbf p}(A)} 1_A(\omega,x) \frac{d\mu_{\mathbf p}}{d\lambda}(x)$. According to Lemma \ref{l:productmeasure} this yields an acs measure $\tilde{\mu}$ for $F$ such that $\rho = \mathbb{P} \times \tilde{\mu}$. Letting $L := \text{supp}(\tilde{\mu})$ denote the support of $\tilde{\mu}$, we obtain that $\text{supp}(\rho) = \Sigma^{\mathbb{N}} \times L$. By definition of $\rho$, it follows that $\Sigma^{\mathbb{N}} \times L \subseteq A$ and
\begin{align*}
A = \Sigma^{\mathbb{N}} \times L \quad \bmod \mathbb{P} \times \mu_{\mathbf p}.
\end{align*}
Since $\frac{d\mu_{\mathbf p}}{d\lambda}$ has full support on $[0,1]$, this yields
\begin{align}\label{eq18a}
A = \Sigma^{\mathbb{N}} \times L \quad \bmod \mathbb{P} \times \lambda.
\end{align}
Using the non-singularity of $F$ with respect to $\mathbb{P} \times \lambda$, we also obtain from this that
\begin{align}\label{eq18b}
F^{-1}A = F^{-1}(\Sigma^{\mathbb{N}} \times L) \quad \bmod \mathbb{P} \times \lambda.
\end{align}
Combining \eqref{eq18a} and \eqref{eq18b} with
\begin{align*}
\Sigma^{\mathbb{N}} \times L = \bigcup_{j \in \Sigma} [j] \times L \quad \text{and} \quad F^{-1}(\Sigma^{\mathbb{N}} \times L) = \bigcup_{j \in \Sigma} [j] \times T_j^{-1} L
\end{align*}
yields
\begin{align*}
L = T_j^{-1} L \quad \bmod \lambda
\end{align*}
for each $j \in \Sigma$. For all $i \in \Sigma$ with $\alpha_i < 1$, in particular for $i \in \Sigma$ with $\alpha_i = \alpha_{\min}$, we have that $T_i$ is ergodic with respect to $\lambda$, see e.g.~\cite[Theorem 5]{Y99}. In particular we have $\lambda(L) \in \{0,1\}$. Together with \eqref{eq18a} this shows that $\mathbb{P} \times \lambda(A) \in \{0,1\}$. Since $\mu_{\mathbf p} \ll \lambda$,  it follows from the assumption $\mathbb{P} \times \mu_{\mathbf p}(A) > 0$ that $\mathbb{P} \times \mu_{\mathbf p}(A) = 1$. We conclude that $F$ is ergodic with respect to $\mathbb{P} \times \mu_{\mathbf p}$.

\medskip
(2) Since $\frac{d\mu_{\mathbf p}}{d\lambda} \in \mathcal{C}_2$, it follows that $\frac{d\mu_{\mathbf p}}{d\lambda}$ is bounded away from zero, is decreasing on the intervals $(0,\frac{1}{2}]$ and $(\frac{1}{2},1]$, and satisfies \eqref{eq1.5b} and \eqref{eq1.6b} with $a_1,a_2 > 0$ as in Lemma \ref{lemma3.4h}. Furthermore, applying the last three inequalities in \eqref{eq3.25c} with $f = \frac{d\mu_{\mathbf p}}{d\lambda}$ yields, for $x,y \in (0,\frac{1}{2}]$ with $x \geq y$,
\begin{align*}
0 \leq \frac{d\mu_{\mathbf p}}{d\lambda}(y) - \frac{d\mu_{\mathbf p}}{d\lambda}(x) & = y^{-d} \Big(y^{d}\frac{d\mu_{\mathbf p}}{d\lambda}(y) - y^{d} \frac{d\mu_{\mathbf p}}{d\lambda}(x)\Big) \\
& \leq y^{-d} \cdot \frac{d\mu_{\mathbf p}}{d\lambda}(x) (x^{d} - y^{d}) \\
& \leq y^{-d} \cdot a_1  \cdot 2^{-d+1+t_1}\cdot d |x-y|
\end{align*}
and likewise applying the last three inequalities in \eqref{eq3.26c} with $f = \frac{d\mu_{\mathbf p}}{d\lambda}$ yields for $x,y \in (\frac{1}{2},1]$ with $x \geq y$,
\begin{align*}
0 \leq \frac{d\mu_{\mathbf p}}{d\lambda}(y) - \frac{d\mu_{\mathbf p}}{d\lambda}(x) & = \Big(y-\frac{1}{2}\Big)^{-d}\Big(\Big(y-\frac{1}{2}\Big)^{d} \frac{d\mu_{\mathbf p}}{d\lambda}(y) - \Big(y-\frac{1}{2}\Big)^{d} \frac{d\mu_{\mathbf p}}{d\lambda}(x)\Big) \\
& \leq \Big(y-\frac{1}{2}\Big)^{-d} \cdot \frac{d\mu_{\mathbf p}}{d\lambda}(x) \Big(\Big(x-\frac{1}{2}\Big)^{d} - \Big(y-\frac{1}{2}\Big)^{d}\Big) \\
& \leq \Big(y-\frac{1}{2}\Big)^{-d} \cdot a_2 \cdot 2^{-d+1+t_2} \cdot d|x-y|
\end{align*}
Hence, $\frac{d\mu_{\mathbf p}}{d\lambda}$ is locally Lipschitz on the intervals $(0,\frac{1}{2}]$ and $(\frac{1}{2},1]$.
\end{proof}

We conclude this section with the proof of Corollary \ref{cor1.1}.

\begin{proof}[Proof of Corollary \ref{cor1.1}]
For each $n \in \mathbb{N}$, let $\mathbf p_n = (p_{n,j})_{j \in \Sigma}$ be a positive probability vector such that $\sup_n \sum_{r \in \Sigma_R} p_{n,r} K_r^{-\alpha_{\min}} < 1$ and assume that $\lim_{n \rightarrow \infty} \mathbf p_n = \mathbf p$ in $\mathbb{R}_+^N$. In order to conclude that $\frac{d\mu_{\mathbf p_n}}{d\lambda}$ converges in $L^1(\lambda)$ to $\frac{d\mu_{\mathbf p}}{d\lambda}$ we will show that each subsequence of $\{\frac{d\mu_{\mathbf p_n}}{d\lambda}\}$ has a further subsequence that converges in $L^1(\lambda)$ to $\frac{d\mu_{\mathbf p}}{d\lambda}$.

Let $\{\mathbf q_k\}$ be a subsequence of $\{\mathbf p_n\}$, and for convenience write $f_k = \frac{d\mu_{\mathbf q_k}}{d\lambda}$ for each $k \in \mathbb{N}$. First of all, observe that from $\sup_n \sum_{r \in \Sigma_R} p_{n,r} K_r^{-\alpha_{\min}} < 1$ and $\lim_{n \rightarrow \infty} \mathbf p_n = \mathbf p$ it follows from the proof of Lemma \ref{lemma3.4h} that there exist sufficiently large $a_1,a_2 > 0$ and $\beta \in (\alpha_{\min},\gamma)$ sufficiently close to $\alpha_{\min}$ such that $\mathcal{C}_2 = \mathcal{C}_2(a_1,a_2,\beta)$ from Lemma \ref{lemma3.4h} contains the sequence $\{f_k\}$. Hence, it follows from Lemma \ref{lemma3.5u} that $\{f_k\}$ has a subsequence $\{f_{k_m}\}$ that converges  with respect to the $L^1(\lambda)$-norm to some $\tilde{f} \in \mathcal{C}_2$. We have
\begin{align*}
\|P_{F,\mathbf p} \tilde{f} - \tilde{f}\|_1 &\leq \|P_{F,\mathbf p} \tilde{f} - P_{F,\mathbf q_{k_m}} \tilde{f}\|_1 + \| P_{F,\mathbf q_{k_m}} \tilde{f} - f_{k_m} \|_1 + \| f_{k_m} - \tilde{f}\|_1 \\
&\leq \sum_{j \in \Sigma} |p_j - q_{k_m,j}| \cdot \|P_{T_j}  \tilde{f}\|_1 + \| P_{F,\mathbf q_{k_m}} \tilde{f} - P_{F,\mathbf q_{k_m}} f_{k_m}\|_1 + \| f_{k_m} - \tilde{f}\|_1\\
&\leq  \sum_{j \in \Sigma} |p_j - q_{k_m,j}| \cdot \|  \tilde{f}\|_1 + 2\| f_{k_m} - \tilde{f}\|_1.
\end{align*}
Since we have $\lim_{m \rightarrow \infty} \mathbf q_{k_m} = \mathbf p$ in $\mathbb{R}_+^N$ and $\lim_{m \rightarrow \infty} \|f_{k_m} - \tilde{f}\|_1 = 0$ we obtain that $P_{F,\mathbf p} \tilde{f}(x) = \tilde{f}(x)$ holds for $\lambda$-a.e.~$x \in [0,1]$. It follows from Theorem 1.1 that $F$ admits only one acs probability measure associated to $\mathbf p$, so we conclude that $\tilde{f} = \frac{d\mu_{\mathbf p}}{d\lambda}$ holds $\lambda$-a.e. Hence, $\{f_{k_m}\}$ converges in $L^1(\lambda)$ to $\frac{d\mu_{\mathbf p}}{d\lambda}$.
\end{proof}

\section{Final remarks}\label{sec4}

In addition to the results of Theorems \ref{thrm1.1a} and \ref{thrm1.1b}, we conjecture that also if $\eta = 1$ then $F$ admits no acs probability measure. A possible approach to prove this might be to again use Kac's Lemma. However, a finer bound for the behaviour near $\frac{1}{2}$ than the one given in \eqref{eq3.7z} should be needed for the proof.

\medskip
The proof of Theorem \ref{thrm1.1b} immediately carries over to the case that $\Sigma_R = \emptyset$ by taking $\beta = 1$, thus recovering the result from \cite{zeegers18} that a random system generated by i.i.d.~random compositions of finitely many LSV maps admits a unique absolutely continuous invariant probability measure if $\alpha_{\min} < 1$ with density as in \eqref{eq1.7b} for some $a > 0$. To show that in case $\Sigma_R = \emptyset$ this density is decreasing and continuous on the whole interval $(0,1]$ similar arguments as in Section \ref{subsec3.2} can be used with the sets $\mathcal{C}_0$, $\mathcal{C}_1$ and $\mathcal{C}_2$ replaced by
\begin{align*}
&\mathcal{K}_0 = \Big\{f \in L^1(\lambda): f \geq 0, \text{$f$ decreasing and continuous on $(0,1]$}\Big\},\\
&\mathcal{K}_1 = \Big\{f \in \mathcal{K}_0: x \mapsto x^{\alpha_{\max}+1} f(x) \text{ increasing on $(0,1]$}\Big\},\\
&\mathcal{K}_2 = \Big\{ f \in \mathcal{K}_1: f(x) \leq a x^{-\alpha_{\min}} \text{ on $(0,1]$}, \int_0^1 f d\lambda = 1 \Big\} \quad \text{ with $a > 0$ large enough.}
\end{align*}
This has been done in \cite{zeegers18}.

\medskip
It would be interesting to study further statistical properties of the random systems from Theorem \ref{thrm1.1b}. It is proven in \cite{hu04,LSV,Y99,gouezel04} that for $\alpha \in (0,1)$ correlations under $S_{\alpha}$ decay polynomially fast with rate $n^{1-1/\alpha}$. Moreover, for the random system of LSV maps $S_{\alpha}$ considered in \cite{BBD14} where $\alpha$ is sampled i.i.d.~from some fixed finite subset $A \subseteq (0,1]$ the authors obtained that annealed correlations decay as fast as $n^{1-1/\alpha_{\min}}$, a rate that in \cite{BB16} is shown to be sharp for a class of observables that vanish near zero. Here the minimal value $\alpha_{\min}$ of $A$ is assumed to lie in $(0,1)$. This result was extended in \cite{BQT21} to the case that $A \subseteq (0,\infty)$ is not necessarily a finite subset of $(0,1)$ and there is a positive probability to choose a parameter $< 1$. The results from \cite{BBD14,BB16,zeegers18,BQT21} demonstrate that 
the annealed dynamics of such random systems of LSV maps are governed by the map with the fastest relaxation rate. This behaviour is significantly different from the behaviour of the random systems from Theorem \ref{thrm1.1b} where the annealed dynamics is determined by the interplay between the neutral exponentially fast attraction to $\frac{1}{2}$ and polynomially fast repulsion from zero. We conjecture that the random systems from Theorem \ref{thrm1.1b} are mixing and that in case $\Sigma_R = \{1\}$ the rate of the annealed decay of correlations is equal to $1- \frac{1}{\alpha_{\min}} \cdot \min\{\frac{\log p_1}{\log K_1},1\}$.

\medskip
Finally, a natural question is whether the results of Theorems \ref{thrm1.1a} and \ref{thrm1.1b} can be extended to a more general class of one-dimensional random dynamical systems that exhibit this interplay between two fixed points, one to which orbits converge exponentially fast and one from which orbits diverge polynomially fast. First of all, if being $C^1$ and having $\frac{1}{2}$ as attracting fixed point are the only conditions we put on the right branches of the maps in $\mathfrak{R}$, then it can be shown in a similar way as in the proof of Theorem \ref{thrm1.1a} that $F$ admits no acs probability measure if
\begin{align*}
\sum_{r \in \Sigma_R} p_r \cdot \Big(\lim_{x \downarrow \frac{1}{2}} |DT_r(x)|\Big)^{-\alpha_{\min}} > 1
\end{align*}
by applying Kac's Lemma on a small enough domain in $\Sigma^{\mathbb{N}} \times (0,\frac{1}{2})$. 
Secondly, we used in the proofs of Lemma \ref{lemma3.3} and Lemma \ref{lemma3.4h} that $\frac{1}{DR_{\alpha_r,K_r}}\big(\frac{x-\frac{1}{2}}{z_r-\frac{1}{2}}\big)^{d}$ is increasing and $\frac{1}{DR_{\alpha_r,K_r}}\big(\frac{x-\frac{1}{2}}{z_r-\frac{1}{2}}\big)^{t_2}$ is decreasing, respectively, by means of the results on $H_{K,b}$ in Lemma \ref{lemma3.2b}. However, for other maps that have the property that $\frac{1}{2}$ and $1$ are fixed points and that orbits are attracted to $\frac{1}{2}$ exponentially fast this is not true in general. Still a phase transition is to be expected, but different techniques are needed to prove this. This is also the case when we drop the condition that $1$ is a fixed point of the maps in $\mathfrak{R}$, for instance by taking $R_{\alpha,K}(x) = \frac{1}{2}+K(x-\frac{1}{2})$ if $x \in (\frac{1}{2},1]$, in which case Lemma \ref{lemma3.3} would not hold. Thirdly, the results of Theorems \ref{thrm1.1a} and \ref{thrm1.1b} might carry over if we allow the left branches to only satisfy the conditions on the left branch of the maps $\{T_{\alpha}: [0,1] \rightarrow [0,1]\}_{\alpha \in (0,1)}$ considered in \cite{murray05} or Section 5 of \cite{LSV}. Each map $T_{\alpha}$ then satisfies $T_{\alpha}(0) = 0$ and $DT_{\alpha}(x) = 1 + C x^{\alpha} + o(x^{ \alpha})$ for $x$ close to zero and where $C > 0$ is some constant.

\bibliographystyle{plain} 
\bibliography{Bibliography}

\end{document}